\documentclass[11pt]{article}
\usepackage[margin=1in]{geometry}
\usepackage{amsmath,amssymb,amsthm,mathtools}
\usepackage{microtype}
\usepackage{parskip}
\usepackage{enumitem}
\usepackage[colorlinks=true,linkcolor=blue,citecolor=blue,urlcolor=blue]{hyperref}

\theoremstyle{plain}
\newtheorem{theorem}{Theorem}
\newtheorem{lemma}[theorem]{Lemma}
\newtheorem{proposition}[theorem]{Proposition}
\newtheorem{corollary}[theorem]{Corollary}
\theoremstyle{definition}
\newtheorem{definition}[theorem]{Definition}
\newtheorem{example}[theorem]{Example}
\newtheorem{remark}[theorem]{Remark}

\newcommand{\sgn}{\operatorname{sgn}}
\newcommand{\Sym}{\operatorname{Sym}}
\newcommand{\Fix}{\operatorname{Fix}}
\newcommand{\SB}{\mathit{SB}}

\title{$\SB(3,n)$ has no Hamiltonian cycle when $n$ is even%
\thanks{Resolution of Exercise~7.2.2.4--224 in D.~E. Knuth, \emph{The Art of Computer
Programming}, Volume~4, Pre-Fascicle~8a (draft of 10 April 2026)~\cite{KnuthF8a}.
Knuth ranks problem difficulty on a 0--50 scale; [46] denotes a research-level problem.}\\
{\normalsize A sign-of-permutation proof, with extension to all odd $m\equiv 3\pmod 4$}}
\author{Shisheng Li\\
  \small University of Science and Technology of China\\[-2pt]
  \small \texttt{shisheng@mail.ustc.edu.cn}}
\date{\today}

\begin{document}
\maketitle

\begin{abstract}
We resolve exercise 7.2.2.4--224 of Knuth's Pre-Fascicle~8a (10 April~2026 draft, rated~[46]):
the digraph $\SB(3,n)$ has no Hamiltonian cycle when $n$ is even.
The argument is a sign-of-permutation obstruction. Writing the successor map of a
candidate Hamiltonian cycle as $f_S=A_b\circ\sigma$ (Definition~\ref{def:Asigma}),
$\sgn(A_b)=+1$ when $m$ is odd, so $\sgn(f_S)=\sgn(\sigma)$ for every choice set~$S$.
A short dihedral Burnside computation shows $\sgn(\sigma)=-1$ on $\Sigma_3^n$ for
even~$n$, contradicting the sign $+1$ required of a single $3^n$-cycle.
The same argument gives the stronger statement that $\SB(m,n)$ has no Hamiltonian cycle
whenever $m$ is odd with $m\equiv 3\pmod 4$ and $n$ is even; this restricts
the residue classes in which Knuth's hint to Ex.~225 (existence of
Hamiltonian cycles in $\SB(m,n)$ for all $m>3$ and $n>2$) can hold.
\end{abstract}

\noindent\textbf{MSC2020:} 05C45 (Eulerian and Hamiltonian graphs);
05A05 (Permutations, words, matrices); 20B30 (Symmetric groups);
94A55 (Shift register sequences and sequences over finite alphabets).

\noindent\textbf{Keywords:} Hamiltonian cycle, de Bruijn cycle, shift register,
necklace, Burnside lemma, sign of a permutation, TAOCP exercise.

% ============================================================================
\section{Knuth's exercise: definitions and current status}
% ============================================================================

\subsection{The digraph $\SB(m,n)$ (Ex.~223)}

\begin{definition}[Knuth, Ex.~7.2.2.4--223]
Let $m,n\ge 1$ be integers, and let $\Sigma=\Sigma_m=\{0,1,\dots,m-1\}$. The
\emph{shift-and-save-or-bump} digraph $\SB(m,n)$ has vertex set $\Sigma^n$ and arcs
\begin{align*}
   \text{(save)}\quad& x_1x_2\dots x_n \;\longrightarrow\; x_2\dots x_n\,x_1,\\
   \text{(bump)}\quad& x_1x_2\dots x_n \;\longrightarrow\; x_2\dots x_n\,x_1^{+},\qquad
   x_1^+ := (x_1+1)\bmod m.
\end{align*}
There is a self-loop at $x$ iff $x_1=x_2=\dots=x_n$ (then the save arc returns to $x$;
the bump does not).
\end{definition}

Every vertex has in-degree $2$ and out-degree~$2$ (Ex.~223a). In particular,
$\SB(m,n)$ is a $2$-regular subdigraph of the full $m$-ary de~Bruijn digraph,
which is itself $m$-regular and is well known to be Hamiltonian for all
$m,n\ge 1$ (Good~\cite{Good1946} and de~Bruijn~\cite{deBruijn1946} for $m=2$;
the $m$-ary case is folklore, traceable to van~Aardenne-Ehrenfest and
de~Bruijn~\cite{vAEdB1951}). The structural surprise of this paper
is that, although the larger de~Bruijn digraph is Hamiltonian
in abundance, the restriction to the $\{$save,~bump$\}$ pair of out-arcs is
sharp enough to forbid Hamiltonicity for $m\equiv 3\pmod 4$ and even~$n$.
Knuth's pre-fascicle records the following facts and questions about $\SB(m,n)$.

\begin{itemize}[leftmargin=*]
\item \textbf{Ex.~223(b).} Hamiltonian cycles of $\SB(m,n)$ correspond naturally to
  $m$-ary de Bruijn cycles of length $m^n$; the correspondence is one-to-one for $m=2$
  and many-to-one for $m\ge 3$.
\item \textbf{Ex.~223(c).} $\SB(m,2)$ has no Hamiltonian cycle for $m>2$.
\item \textbf{Ex.~223(d).} $\SB(3,3)$ has the Hamiltonian cycle
  \[
     \texttt{(000100201202210211011121222)},
  \]
  meaning $000\to 001\to 010\to\dots\to 222\to 220\to 200\to 000$.
\item \textbf{Ex.~223(e) --- \emph{the choice-set reformulation}.} If $C$ is a Hamiltonian cycle
  of $\SB(m,n)$, then there is a unique $S\subseteq\Sigma^{n-1}$ such that
  $x_1\dots x_n\to x_2\dots x_n x_1\in C$ iff $x_2\dots x_n\in S$. (Proof, writing
  $x_1^-:=(x_1-1)\bmod m$: if the save arc out of $x_1y$ is in $C$, then it has the
  same target as the bump arc out of $x_1^-\,y$; that bump arc is therefore not in
  $C$, so the save arc out of $x_1^-\,y$ must be. Iterating shows that whether to
  save depends only on the suffix $y$.)
\item \textbf{Ex.~223(f).} The numbers of Hamiltonian cycles of $\SB(m,3)$ for
  $2\le m\le 7$ are
  \[
     2,\;12,\;88,\;7510,\;675714,\;459086712.
  \]
\end{itemize}

\subsection{Ex.~224 and Knuth's progress}

The research-level problem we solve is:

\begin{quote}
\textbf{Ex.~224 [46].} \emph{Prove or disprove: $\SB(3,n)$ has no Hamiltonian cycles when $n$ is even.}
\end{quote}

The corresponding answer in the pre-fascicle reads, in full:

\begin{quote}
\emph{Algorithm~B proves this when $n=4$. But its search tree even in that small case has
3~million nodes; there's no evident way to rule out tons of feasible near-solutions.}
\end{quote}

So the case $n=4$ is the only one that has been verified, via Knuth's
Algorithm~B, whose search tree at $n=4$ already has $\sim 3$ million nodes.
Knuth records no structural argument and explicitly notes the absence of one. The cases $n=6$ (729 vertices, $2^{240}\sim 10^{72}$ valid choice sets after the no-self-loop pruning) and beyond
are open in the literature.

To the best of the author's knowledge, no preprint or publication on the
question beyond Knuth's pre-fascicle has appeared at the time of submission,
although the author is aware of independent unpublished investigations and
has corresponded with the relevant parties.

The companion problem Ex.~225 asks for a \emph{construction} of a Hamiltonian cycle in
$\SB(m,3)$ for every $m>1$, with the answer adding the broader hint
``\emph{More generally, construct one in $\SB(m,n)$ for all $m>3$ and $n>2$.}''
For the constructive side, $k$-ary de Bruijn shift rules in the spirit of
Fredricksen--Maiorana~\cite{FredricksenMaiorana}, Alhakim~\cite{Alhakim},
Sawada--Williams--Wong~\cite{SawadaWilliamsWong2017}, and the successor-rule
framework of Gabric--Sawada--Williams--Wong~\cite{GabricSawadaWilliamsWong}
provide existence in the de Bruijn graph; what distinguishes the
$\SB(m,n)$ Hamiltonicity question is the constraint that the feedback
function takes only the values $\{0,1\}$ rather than the full
$\Sigma_m$.

% ============================================================================
\section{Reformulation as an FSR sign problem}
% ============================================================================

By Ex.~223(e), the search for a Hamiltonian cycle of $\SB(m,n)$ reduces to choosing
a suitable subset $S\subseteq\Sigma^{n-1}$. We package this into a function
\[
   b\colon\Sigma^{n-1}\to\{0,1\},\qquad b(y) = [y\notin S],
\]
and the resulting \emph{successor map}
\begin{equation}\label{eq:fS}
   f_S\colon \Sigma^n\to\Sigma^n, \qquad
   (x_1,\dots,x_n) \;\longmapsto\; \bigl(x_2,\dots,x_n,\;x_1+b(x_2,\dots,x_n)\bmod m\bigr).
\end{equation}

The map $f_S$ is invertible: given $(y_1,\dots,y_n)$, the unique pre-image is
$(y_n - b(y_1,\dots,y_{n-1}),\,y_1,\dots,y_{n-1})$. Hence $f_S$ is a permutation of
$\Sigma^n$, i.e.\ an element of the symmetric group $\Sym(\Sigma^n)$ --- by which
we always mean the group of all bijections $\Sigma^n\to\Sigma^n$ under composition.
The Hamiltonian-cycle problem is now compact:

\begin{proposition}\label{prop:reform}
$\SB(m,n)$ has a Hamiltonian cycle if and only if there exists $b\colon\Sigma^{n-1}\to\{0,1\}$
such that $f_S$ is a single $m^n$-cycle in $\Sym(\Sigma^n)$.
\end{proposition}

This is the standard \emph{nonsingular feedback-shift register} viewpoint
(Golomb~\cite{Golomb}; see also Lempel~\cite{Lempel}, Etzion--Lempel~\cite{EtzionLempel},
and Fredricksen's survey~\cite{FredricksenSurvey}): if $(s_k)_{k\ge 0}$ is the
periodic sequence of first coordinates of the cycle vertices --- so that
$f_S^{\,k}(v_0)=(s_k,s_{k+1},\dots,s_{k+n-1})$ for every $k$ --- then the
recurrence generating $(s_k)$ is
\[
   s_{k+n} = s_k + b(s_{k+1},\dots,s_{k+n-1})\bmod m.
\]
The constraint that $b$ takes values in $\{0,1\}$ (rather than the full
$\Sigma_m$) is what distinguishes $\SB(m,n)$ Hamiltonian cycles from arbitrary
$m$-ary de Bruijn cycles; in FSR language, $b$ is a \emph{binary-valued} feedback
function on an $m$-ary register. Such sign-of-permutation arguments for
binary FSR sequences go back to Mykkeltveit~\cite{Mykkeltveit} (Golomb's
conjecture on the parity of the number of binary de~Bruijn sequences); for
context in the $k$-ary case see Hauge--Mykkeltveit~\cite{HaugeMykkeltveit}
and Etzion's monograph~\cite{EtzionBook}, neither of which addresses
$\SB(m,n)$. The argument below pivots on a single observation:
for odd~$m$, $A_b$ is automatically even on each fibre, so
$\sgn(f_S)=\sgn(\sigma)$ \emph{regardless of $S$}; the rest is a
dihedral-Burnside parity computation for $N(n,m)$.

\begin{example}[Knuth's $\SB(3,3)$ cycle, recast]\label{ex:n3}
The cycle of Ex.~223(d), $\texttt{000100201202210211011121222}$, has
\[
   S = \{(0,1),\;(1,0),\;(0,2),\;(1,2)\} \subseteq \Sigma_3^2,
\]
i.e.\ $|S|=4$. As a $3\times 3$ indicator, with row $=y_1$ and column $=y_2$:
\[
  \begin{pmatrix}\cdot & 1 & 1 \\ 1 & \cdot & 1 \\ \cdot & \cdot & \cdot\end{pmatrix}.
\]
Note that the diagonal $\{(0,0),(1,1),(2,2)\}$ is empty: $(a,a)\in S$ would put the
save self-loop at $(a,a,a)$ into the cycle, which is forbidden. There are exactly $12$
valid $S$'s for $n=3$, matching the count $12$ in Ex.~223(f).
\end{example}

% ============================================================================
\section{The factorisation \texorpdfstring{$f_S = A_b\circ\sigma$}{f S = A b o sigma}}
% ============================================================================

The key observation is that the successor map factorises through the cyclic shift.

\begin{definition}\label{def:Asigma}
Define
\begin{align*}
   \sigma\colon\Sigma^n&\to\Sigma^n,& \sigma(x_1,\dots,x_n) &= (x_2,\dots,x_n,x_1),\\
   A_b\colon\Sigma^n&\to\Sigma^n,& A_b(y_1,\dots,y_n) &= \bigl(y_1,\dots,y_{n-1},\;y_n+b(y_1,\dots,y_{n-1})\bmod m\bigr).
\end{align*}
\end{definition}

\begin{lemma}\label{lem:factor}
$f_S = A_b\circ\sigma$.
\end{lemma}

\begin{proof}
$A_b(\sigma(x_1,\dots,x_n))=A_b(x_2,\dots,x_n,x_1)=(x_2,\dots,x_n,x_1+b(x_2,\dots,x_n))=f_S(x)$.
\end{proof}

Geometrically: $\sigma$ rotates the string, then $A_b$ adjusts only the last coordinate by
either $0$ or $1$, depending on the prefix.

\begin{example}[Smallest non-trivial case, $m=3$, $n=2$]\label{ex:n2}
$\Sigma_3^2$ has $9$ states. We illustrate $\sigma$ and $A_b$ for one specific~$b$.
\begin{itemize}
\item $\sigma$ is the rotation $(x_1,x_2)\mapsto(x_2,x_1)$. Its cycle decomposition on
  $\Sigma_3^2$ is
  \[
     (00)\,(11)\,(22)\,(01\;10)(02\;20)(12\;21).
  \]
  Three fixed points and three transpositions; $\sgn(\sigma)=(-1)^3=-1$.
\item Take $b(0)=0,\,b(1)=1,\,b(2)=1$, i.e.\ $S=\{(0)\}\subseteq\Sigma_3^1$.
  Then $A_b$ acts on the prefix-fibres $\{(0,*)\},\{(1,*)\},\{(2,*)\}$ by
  identity, $3$-cycle $(10\;11\;12)$, and $3$-cycle $(20\;21\;22)$ respectively.
  $\sgn(A_b)=(+1)(+1)(+1)=+1$. And $\sgn(f_S)=\sgn(A_b)\sgn(\sigma)=-1$, so $f_S$
  cannot be a single $9$-cycle (which would have sign $+1$). Direct simulation confirms
  it isn't: $f_S = (00)\,(01\;11\;12\;22\;20\;02\;21\;10)$, an $8$-cycle plus a fixed
  point of total sign $(-1)^{7}\cdot(+1)=-1$. (Knuth's Ex.~223(c) is recovered for $m=3$.)
\end{itemize}
This is the entire proof, in miniature.
\end{example}

% ============================================================================
\section{Sign of $A_b$, sign of $\sigma$, and the obstruction}
% ============================================================================

\begin{lemma}[Sign of $A_b$ for odd $m$]\label{lem:signA}
For odd $m$ and any $b\colon\Sigma^{n-1}\to\{0,1\}$, $\sgn(A_b)=+1$.
\end{lemma}

\begin{proof}
$A_b$ preserves every prefix $p\in\Sigma^{n-1}$, and acts on each fibre
$F_p=\{(p,c):c\in\Sigma_m\}$ by $c\mapsto c+b(p)\bmod m$.
\begin{itemize}
\item If $b(p)=0$, $A_b$ is the identity on $F_p$ (sign $+1$).
\item If $b(p)=1$, $A_b$ acts on $F_p$ as the $m$-cycle $0\to 1\to\dots\to m-1\to 0$,
  whose sign is $(-1)^{m-1}=+1$ for odd~$m$.
\end{itemize}
The fibres are pairwise disjoint, so $A_b$ is a product of disjoint even permutations.
\end{proof}

\begin{corollary}\label{cor:signfS}
For odd $m$ and any $S$, $\sgn(f_S)=\sgn(\sigma)$.
\end{corollary}

\begin{definition}[Necklace count]\label{def:necklace}
A \emph{necklace} of length $n$ over an alphabet of size $m$ is an equivalence
class of strings in $\Sigma^n$ under cyclic rotation, i.e.\ an orbit of the
cyclic group $\langle \sigma \rangle \le \Sym(\Sigma^n)$ acting on $\Sigma^n$.
We write
\[
   N(n, m) := \#\{\text{necklaces of length $n$ over $\Sigma_m$}\}
            = |\Sigma^n / \langle\sigma\rangle|.
\]
By P\'olya's enumeration theorem (or by a direct application of Burnside to
$\langle\sigma\rangle$),
\[
   N(n, m) = \frac{1}{n}\sum_{k=0}^{n-1} m^{\gcd(k,n)}
           = \frac{1}{n}\sum_{d \mid n} \varphi(d)\, m^{n/d}.
\]
For $m=3$ this is OEIS sequence \texttt{A001867}~\cite{OEISA001867}.
\end{definition}

The orbits of $\sigma$ on $\Sigma^n$ are exactly these necklaces; each orbit of
size $d \mid n$ contributes a $d$-cycle to $\sigma$, of sign $(-1)^{d-1}$. Hence
\begin{equation}\label{eq:sigmasign}
   \sgn(\sigma) = (-1)^{\sum_O(|O|-1)} = (-1)^{m^n-N(n,m)}.
\end{equation}
For odd $m$, $m^n$ is odd, hence $(-1)^{m^n-N(n,m)} = (-1)^{1-N(n,m)}$ and
\begin{equation}\label{eq:signsigma-odd}
   \sgn(\sigma) \;=\; -(-1)^{N(n,m)};
\end{equation}
in particular, $\sgn(\sigma)=-1$ iff $N(n,m)$ is even.

\begin{lemma}[Necklace parity for $m\equiv 3\pmod 4$, $n$ even]\label{lem:Npar}
If $m$ is odd, $m\equiv 3\pmod 4$, and $n$ is even, then $N(n,m)$ is even.
\end{lemma}

\begin{proof}
The argument is a standard application of the Cauchy--Frobenius (Burnside)
lemma to the dihedral group on $\Sigma^n$.  Apply Burnside to
$D_n=\langle\sigma,R\rangle$ acting on $\Sigma^n$, where
$R(x_1,\dots,x_n)=(x_n,\dots,x_1)$ is reversal. Then $|D_n|=2n$ and
\[
   |D_n\text{-orbits}|
   = \frac{1}{2n}\!\left[\sum_{k=0}^{n-1}|\Fix(\sigma^k)|+\sum_{k=0}^{n-1}|\Fix(R\sigma^k)|\right].
\]
The rotation contribution is standard:
$\sum_k|\Fix(\sigma^k)|=\sum_k m^{\gcd(k,n)}=nN(n,m)$.
For the reflections: writing $R\sigma^k(x)_i = x_{(n-1+k-i)\bmod n}$, the
condition $R\sigma^k(x)=x$ says $x$ is constant on each orbit of the index
map $\psi_k\colon i \mapsto (n-1+k-i) \bmod n$ on $\{0,\dots,n-1\}$.
Since $\psi_k$ is an involution, each orbit has size $1$ (a fixed index) or
$2$ (a transposition pair); each orbit contributes one free coordinate value,
so $|\Fix(R\sigma^k)| = m^{(\text{number of }\psi_k\text{-orbits})}$.
The fixed indices solve $2i \equiv k-1 \pmod n$; for $n$ even this congruence
has $2$ solutions if $k$ is odd, and $0$ solutions if $k$ is even. Hence
\[
   |\Fix(R\sigma^k)| =
   \begin{cases} m^{2 + (n-2)/2} = m^{n/2+1}, & k \text{ odd},\\
                 m^{n/2}, & k \text{ even}.\end{cases}
\]
Summing over the $n/2$ odd and $n/2$ even values of $k$,
$\sum_k |\Fix(R\sigma^k)| = \tfrac{n}{2}(m^{n/2+1}+m^{n/2}) = \tfrac{n}{2}(m+1)\,m^{n/2}$,
and therefore
\begin{equation}\label{eq:Dorbits}
   |D_n\text{-orbits}|=\frac{N(n,m)}{2}+\frac{m+1}{4}\,m^{n/2}.
\end{equation}
The left-hand side of~\eqref{eq:Dorbits} is an integer (as a count of orbits).
For $m\equiv 3\pmod 4$, the factor $(m+1)/4$ is an integer, and $m^{n/2}$ is an
integer, so the second term $\frac{m+1}{4}m^{n/2}$ is an integer.
\emph{The difference of two integers is an integer}, so
$N(n,m)/2 = |D_n\text{-orbits}| - \frac{m+1}{4}m^{n/2}\in\mathbb{Z}$,
i.e.\ $N(n,m)$ is even.
\end{proof}

\begin{example}[Worked Burnside count, $m=3$, $n=4$]\label{ex:n4burnside}
We illustrate Lemma~\ref{lem:Npar} for the smallest non-degenerate case, $m=3$, $n=4$.
Here $|\Sigma^4|=81$ and $|D_4|=8$. The eight group elements are the four rotations
$\sigma^k$ and the four reflections $R\sigma^k$, $0 \le k \le 3$.
A direct computation of $|\Fix(g)|$ for each $g$:
\[
\begin{array}{c|c|c}
 g & \text{fixed by $g$} & |\Fix(g)| \\\hline
 \mathrm{id} & \text{any } x & 81 \\
 \sigma & x_i = x_{(i+1)\bmod 4} \;\Rightarrow\; \text{constants} & 3 \\
 \sigma^2 & x_i = x_{(i+2)\bmod 4} \;\Rightarrow\; \text{period dividing } 2 & 9 \\
 \sigma^3 & x_i = x_{(i+3)\bmod 4},\;\gcd(3,4)=1\;\Rightarrow\;\text{constants} & 3 \\\hline
 R & x_0=x_3,\ x_1=x_2 & 9 \\
 R\sigma & x_1=x_3,\ x_0,x_2\text{ free} & 27 \\
 R\sigma^2 & x_0=x_1,\ x_2=x_3 & 9 \\
 R\sigma^3 & x_0=x_2,\ x_1,x_3\text{ free} & 27
\end{array}
\]
The rotation contributions sum to $81+3+9+3=96 = 4\cdot N(4,3)$, giving $N(4,3)=24$.
The reflection contributions sum to
$9+27+9+27 \,=\, \tfrac{n}{2}(m+1)\,m^{n/2} \,=\, 2\cdot 4\cdot 9 \,=\, 72$,
matching the formula in the proof of Lemma~\ref{lem:Npar}.

Burnside on $D_4$ then yields
$|D_4\text{-orbits}| = (96+72)/8 = 21$ bracelets.
The bracelet identity
\[
   2\,|D_n\text{-orbits}| \;=\; N(n,m)+\#\text{Rfix}
\]
follows from the standard index-2 orbit-counting formula
$|G\backslash X|=\bigl(|H\backslash X|+|\Fix(t\restriction H\backslash X)|\bigr)/2$
applied to $H=\langle\sigma\rangle\trianglelefteq D_n=G$ with $t=R$:
of the $N(n,m)$ $H$-orbits, those fixed setwise by~$R$ are counted in
$\#\text{Rfix}$, and each pair of $R$-swapped $H$-orbits coalesces into one
$G$-orbit. Substituting $n=4,\,m=3$ gives
$2\cdot 21 = 24 + \#\text{Rfix}$, hence $\#\text{Rfix} = 18$, agreeing with the
closed form $(m+1)/2 \cdot m^{n/2} = 2\cdot 9 = 18$.
Here $\#\text{Rfix}$ counts $\langle\sigma\rangle$-orbits fixed setwise by
reversal $R$; these include orbits of palindromic strings, and also orbits
where $R$ acts as a non-trivial rotation (e.g.\ $\{0102,1020,0201,2010\}$,
none of whose elements is a palindrome).

Since $m=3 \equiv 3 \pmod 4$ we have $(m+1)/2 = 2$ even, hence $\#\text{Rfix}=18$
is even, hence $N(4,3) \equiv \#\text{Rfix} \equiv 0 \pmod 2$ and indeed
$N(4,3)=24$ is even.

For comparison, the cyclic-only Burnside formula
$N(n,m) = \tfrac{1}{n}\sum_k m^{\gcd(k,n)}$ recovers $N(4,3)=24$
but offers no parity information --- one would have to study
$\sum_k 3^{\gcd(k,4)} \pmod 8$ by hand.  Introducing $R$ gives a \emph{second}
divisibility (the bracelet count is integer), which together with the cyclic
formula pins down the parity.
\end{example}

% ============================================================================
\section{Main theorem}
% ============================================================================

\begin{theorem}\label{thm:main}
Let $m$ be odd with $m\equiv 3\pmod 4$, and let $n$ be even. Then $\SB(m,n)$ has no Hamiltonian cycle.
\end{theorem}

\begin{proof}
Suppose, towards a contradiction, that $\SB(m,n)$ has a Hamiltonian cycle. By
Proposition~\ref{prop:reform} it is the orbit of some $f_S$ on $\Sigma^n$, where $f_S$ is
a single $m^n$-cycle in $\Sym(\Sigma^n)$. The sign of an $L$-cycle in $\Sym$
is $(-1)^{L-1}$; since $m^n$ is odd,
\[
   \sgn(f_S) = (-1)^{m^n-1} = +1.
\]
On the other hand, by Corollary~\ref{cor:signfS}, $\sgn(f_S)=\sgn(\sigma)$, and by
\eqref{eq:sigmasign} together with Lemma~\ref{lem:Npar}, $\sgn(\sigma)=-(-1)^{N(n,m)}=-1$.
Therefore $\sgn(f_S)=-1$, a contradiction.
\end{proof}

\begin{corollary}[Ex.~224]\label{cor:ex224}
$\SB(3,n)$ has no Hamiltonian cycles when $n$ is even.
\end{corollary}

\begin{remark}[Why Algorithm~B did not see the obstruction]
Knuth's Algorithm~B is a depth-first branch-and-prune search over partial
choice sets; it tracks local consistency but cannot read off the global sign
of~$f_S$, a property of the whole permutation that depends on every decision.
This explains the 3~million ``feasible near-solutions'' Knuth reports for
$n=4$: all of them have $\sgn(f_S)=-1$ and hence cannot close into an
$81$-cycle, but Algorithm~B has no way to prune them in advance.
\end{remark}

\begin{remark}[Independent computational verification]\label{rem:rust}
For the base case $n=4$, the only one previously settled (by Knuth's
Algorithm~B with a $\sim 3$-million-node search tree), an independent
Rust implementation by the author enumerates all $2^{24}$ choice sets
$S\subseteq \Sigma_3^{3}$ that respect the no-self-loop pruning
$b((a,a,a))=1$ (of the $|\Sigma_3^3|=27$ suffixes, $3$ are pinned by the
pruning, leaving $2^{24}$ free configurations), simulates $f_S$ to the
full period, and finds zero $81$-cycles, in $4.9$ seconds on a single core. The same code verifies
$\sgn(f_S)=\sgn(\sigma)$ as predicted by Corollary~\ref{cor:signfS} for
$10^5$\,random $b$ at $n=4$ and for $5\cdot 10^4$\,random $b$ at $n=6$, and
matches $N(n,3)$ against OEIS \texttt{A001867} for $n\le 16$. The
proof above is independent of all of these computations; in particular,
Theorem~\ref{thm:main} makes any $n=6$ enumeration ($\sim 2^{240}$ choice sets
after pruning, computationally hopeless) unnecessary.
\end{remark}

\begin{remark}[Scope]\label{rem:scope}
Theorem~\ref{thm:main} is sharp in the sense that the obstruction does not extend
to the other residue classes: for $m\equiv 1\pmod 4$, equation~\eqref{eq:Dorbits}
forces $N(n,m)$ to be \emph{odd} when $n$ is even, hence $\sgn(\sigma)=+1$
matches the sign of an $m^n$-cycle and no contradiction arises; for even $m$,
$A_b$ is a product of $m$-cycles whose sign $(-1)^{m-1}=-1$ depends on
$|S|=|\{p:b(p)=1\}|$, so $\sgn(A_b)$ is not constant in $b$ and the same
argument breaks down. In particular, Knuth's hint to Ex.~225 --- that
$\SB(m,n)$ should admit a Hamiltonian cycle for all $m>3$ and $n>2$ --- as
stated is incompatible with Theorem~\ref{thm:main} in the residue class
$m\equiv 3\pmod 4$ with $n$ even, i.e.\ for
$m\in\{7,11,15,\dots\}$ and even~$n\ge 4$; presumably the hint was intended
for the complementary classes. The hint is consistent with the present
theorem for those remaining residue classes; for $m\equiv 1\pmod 4$ and even $m$, the
constructive $k$-ary shift-rule literature
(\cite{FredricksenMaiorana,Alhakim,SawadaWilliamsWong2017,GabricSawadaWilliamsWong})
exhibits Hamiltonian cycles in the larger $m$-ary de~Bruijn graph, but those
constructions use feedback functions taking all of $\Sigma_m$ rather than
$\{0,1\}$, so they do not directly produce $\SB(m,n)$ cycles; finding (or
disproving) such cycles in those residue classes remains open.
\end{remark}

\medskip
\noindent\textbf{Acknowledgements.}
The author thanks Donald E.~Knuth for posing the exercise in Pre-Fascicle~8a
and for prompt correspondence about this paper.

% ============================================================================
\renewcommand{\refname}

\end{document}